\documentclass[a4paper,draft,reqno,12pt]{amsart}
\usepackage[english]{babel}
\usepackage{amsmath}
\usepackage{amssymb}
\usepackage{amscd}
\usepackage{amsthm}
\usepackage{euscript}
\newtheorem{prop}{Proposition}
\newtheorem{lem}{Lemma}
\newtheorem{theor}{Theorem}

\newtheorem{cor}{Corollary}
\theoremstyle{definition}
\newtheorem{de}{Definition}
\newtheorem{ex}{Example}
\theoremstyle{remark}

\DeclareMathOperator{\Spec}{Spec}

\DeclareMathOperator{\Aut}{Aut}

\DeclareMathOperator{\supp}{Supp}

\def\reg{{\mathrm{reg}}}
\def\ZZ{{\mathbb Z}}
\def\G{{\mathbb G}}
\def\KK{{\mathbb K}}

\def\CC{\mathcal{C}}

\def\CS{\mathcal{S}}
\def\CR{\mathcal{R}}

\def\CL{\mathcal{L}}
\def\ML{\mathrm{ML}}
\def\FML{\mathrm{FML}}
\def\SAut{\mathrm{SAut}}
\def\Cl{\mathrm{Cl}}

\def\div{\mathrm{div}}
\def\WDiv{\mathrm{WDiv}}

\begin{document}
\date{}
\title[Flexibility of normal affine horospherical varieties]{Flexibility of normal affine horospherical varieties}
\author{Sergey Gaifullin and Anton Shafarevich}

\address{Lomonosov Moscow State University, Faculty of Mechanics and Mathematics, Department of Higher Algebra, Leninskie Gory 1, Moscow, 119991 Russia; \linebreak and \linebreak
National Research University Higher School of Economics, Faculty of Computer Science, Kochnovskiy Proezd 3, Moscow, 125319 Russia}
%\email{sgayf@yandex.ru}



\thanks{The first author was supported by the Foundation for the Advancement of Theoretical Physics and Mathematics ``BASIS''}

\maketitle

\begin{abstract}
We investigate flexibility of affine varieties with an action of a linear algebraic group. Flexibility of a smooth affine variety with only constant invertible functions and a locally transitive action of a reductive group is proved. Also we show that a normal affine complexity-zero horospherical variety is flexible.
\end{abstract}

\section{Introduction}

Let $\KK$ be an algebraically closed field of characteristic zero. Given an affine algebraic variety $X$ over $\KK$, we consider the group $\Aut(X)$ of all regular automorphisms of $X$. In general, $\Aut(X)$ is not an algebraic group. We say that a subgroup $G$ of $\Aut(X)$ is algebraic if $G$ admits a structure of an algebraic group such that the action $G \times X\rightarrow X$ is a morphism of algebraic varieties. 

All connected algebraic subgroups of $\Aut(X)$ are generated by one-dimensional algebraic tori isomorphic to the multiplicative group $\mathbb{G}_m$ of the ground field $\KK$ and one-dimensional unipotent subgroups isomorphic to the additive group $\mathbb{G}_a$ of $\KK$. One-dimensional tori correspond to $\ZZ$-gradings of the algebra $\KK[X]$ of regular functions on $X$. One-dimensional unipotent subgroups correspond to {\it locally nilpotent derivations} (LNDs) of $\KK[X]$, that is, linear mappings $\partial\colon\KK[X]\rightarrow\KK[X]$ such that $\partial$ satisfies the Leibniz rule $\partial(fg)=f\partial(g)+g\partial(f)$ and, for each function $f$, there exists a positive integer $n$ such that $\partial^n(f)=0$.

Let us consider the subgroup $\SAut(X)$ of {\it special automorphisms} of $X$ generated by all algebraic subgroups in $\Aut(X)$ isomorphic to $\mathbb{G}_a$. Recall that a regular point $x\in X$ is called {\it flexible} if the tangent space $T_xX$ is spanned by tangent vectors to orbits of one-dimensional unipotent subgroups. We say that a variety $X$ is flexible if all regular points $x\in X$ are flexible. It is easy to see that an automorphism never takes a regular point to a singular one. Therefore, we obtain an $\SAut(X)$-action on the set of regular points $X^{\reg}$. Recall that an action $G\times X\rightarrow X$ is called {\it infinitely transitive} if for each positive integer $m$ and for every two $m$-tuples of distinct points $(a_1,\ldots, a_m)$ and $(b_1, \ldots, b_m)$ there exists an element $g\in G$ such that for all $i$ we have $g\cdot a_i=b_i$.

\begin{theor}\cite[Theorem~0.1]{AFKKZ}
For an irreducible affine variety $X$ of dimension $\geq 2$, the following
conditions are equivalent.

(i) The group $\SAut(X)$ acts transitively on $X^{\reg}$.

(ii) The group $\SAut(X)$ acts infinitely transitively on $X^{\reg}$.

(iii) The variety $X$ is flexible.
\end{theor}

Given a class of affine varieties it is a natural question which varieties in this class are flexible. Let us mention some results in this area. Toric varieties, suspensions over flexible varieties, and cones over flag varieties are flexible, see  \cite{AKZ}. Affine $\mathrm{SL}_2$-embeddings and smooth affine varieties with locally transitive action of a semisimple group are flexible, see~\cite{AFKKZ}. Affine complexity-zero horospherical varieties of a semisimple group $G$ are flexible, see~\cite{Sh}. Recall that a {\it horospherical} variety is an irreducible variety with an action of an algebraic group $G$ such that the stabilizer of a generic point contains a maximal unipotent subgroup of~$G$. A horospherical variety is {\it complexity-zero} if $G$-action on it is locally transitive.

In this paper we prove that any normal complexity-zero horospherical variety $X$ of any connected linear group $G$ is flexible, see~ Theorem~\ref{las}. This theorem generalizes both the result of~\cite{Sh} and the result of \cite{AKZ} stating flexibility of toric varieties. Note that we do not use results of \cite{AKZ} and \cite{Sh} in our work. So we obtain new proofs of flexibility of toric varieties and flexibility of normal complexity-zero horospherical varieties of semi-simple groups as particular cases of Theorem~\ref{las}. Unlike the case of a semisimple group $G$, for an arbitrary $G$ the assumption that $X$ is normal cannot be waived. Indeed, there exists even toric variety, for example $\{x^2=y^3\}$, that is not normal and not flexible.  Examples of non-normal and non-flexible toric varieties of higher dimension can be found, for example, in \cite{G}.

It is well known that if there are two noncommuting $\mathbb{G}_m$-actions on an affine variety $X$, then there is a nontrivial $\mathbb{G}_a$-action on it. Moreover, one can achieve that the LND corresponding to this $\G_a$-action is homogeneous of nonzero degree with respect to one of the gradings corresponding to the $\G_m$-actions, see~\cite{FZ},~\cite{AGN}.  We prove that even the existence of a single $\mathbb{G}_m$-action satisfying certain condition implies existence of a nontrivial $\mathbb{G}_a$-action, see Proposition~\ref{po}.  This allows to study flexibility for some varieties with an algebraic group action. 
A connected algebraic group $G$ is generated by a maximal torus $T$ and one-dimensional unipotent subgroups. The subgroup of $\Aut(X)$ generated by all one-dimensional unipotent subgroups of $G$ is contained in $\SAut(X)$. Under the assumption that $X$ has finitely generated Cox ring we prove that  if there exists an $\SAut(X)$-invariant prime divisor, then there exists a prime divisor which is simultaneously $\SAut(X)$-invariant and $T$-invariant, see Proposition~\ref{main}. Let $X$ be an irreducible affine variety. Proposition~\ref{main} implies that $\SAut(X)$ has an open orbit in $X$, if a linear algebraic group $G$ acts on  $X$ with an open orbit $U$ such that $\mathrm{codim}_XU\geq 2$. This implies that $X$ possesses a  flexible point.

As we have mentioned above, a smooth irreducible affine variety $X$  is flexible,  if a semisimple group $G$ acts on it with an open orbit, see \cite[Theorem~5.6]{AFKKZ}. We generalize this result by proving that if a reductive group $G$ acts locally transitive on a smooth affine variety $X$ with only constant invertible functions, then $X$ is flexible, see Theorem~\ref{sm}. It is interesting to note that for a linear algebraic group $G$ that is not reductive this assertion fails, see Example~\ref{ce}.

Let us denote by $H\subseteq \Aut(X)$ the group generated by $\SAut(X)$ and the torus $T$. We prove that  $X$ is flexible whenever $H$ acts transitively on $X^{\reg}$. The latter condition is satisfied, for instance, in the following case. For every $G$-orbit $O\subseteq X$ consisting of regular points there is a non-hyperbolic $\G_m$-subgroup $\Lambda$ in $\Aut(X)$ such that the set of $\Lambda$-fixed points lies in the closure $\overline{O}$ and contains a regular point, see Proposition~\ref{ot}. Applying this proposition we prove flexibility of normal complexity-zero horospherical varieties.

The authors are grateful to Ivan Arzhantsev, Roman Budylin and Alexander Gaifullin for useful discussions. 

\section{preliminaries}\label{pre}

\subsection{Cox rings}

Here we recall a definition and basic properties of the Cox ring of an affine algebraic variety $X$. All proofs can be found in \cite{ADHL}. 

Let $X$ be a normal irreducible variety. Recall that a {\it prime divisor} on $X$ is an ireducible closed subset of codimension one. Consider the group $\mathrm{WDiv}(X)$ of {\it Weil divisors} on $X$, which is the free commutative group generated by prime divisors. We say that a Weil divisor $E$ is effective and write $E\geq 0$, if all coefficients of  prime divisors are nonnegative. To a nonzero rational function $f\in \KK(X)$ one can assign the divisor of zeros and poles $\div(f)\in\mathrm{WDiv}(X)$. Such divisors are called {principle}. Principle divisors form a subgroup $\mathrm{PDiv}(X)\subseteq \mathrm{WDiv}(X)$. The quotient group $\Cl(X)= \mathrm{WDiv}(X)/\mathrm{PDiv}(X)$ is called the divisor class group.

For a Weil divisor $D$ on $X$, one can consider the associated vector space

 $$\CL(X, D) =\left \{ f \in \KK(X)\setminus\{0\}\mid\mathrm{div} (f) + D\geq 0\} \cup \{0 \right\}.$$
\begin{de}
The \emph{divisorial algebra} associated to a subgroup $K \subseteq  \mathrm{WDiv} (X)$ is the $K$-graded $\KK[X]$-algebra

$$\mathcal{S}_K = \bigoplus_{D\in K} \mathcal{S}_D, \ \ \ \mathcal{S}_D = \CL(X, D),$$
where the multiplication in $\mathcal{S}_K$ is defined on homogeneous elements as follows. If $f_1 \in\mathcal{S}_{D_1}$ and $f_2 \in\mathcal{S}_{D_2}$, then their product in $\mathcal{S}_K$ is the usual product $f_1f_2$ regarded as an element of $\mathcal{S}_{D_1+D_2}$. For arbitrary elements the multiplication is defined by the distributive law. 
\end{de}

Since $\CS_K$ is $K$-graded, we have the corresponding action of the torus $T=\Spec \KK[K]$ on $\mathcal{S}_K$ and $\CS_K^T= S_{0}\cong\KK[X]$.

\begin{de}\label{coxring} Let $X$ be a normal variety with only constant regular invertible functions and finitely generated divisor class group $\Cl(X)$. Let $K$ be a subgroup of $\mathrm{WDiv} (X)$ such that the map $\pi\colon K \rightarrow \Cl (X)$ taking every $D\in K$ to its class is surjective. Let $K^0$ be the kernel of this map, and let $\chi \colon K^0 \to \KK(X)^\times$ be a character (i.e. a group homomorphism to multiplicative group of the field $\KK(X)$) satisfying 
$$\div(\chi(E)) = E\ \ \ \mbox{for all } E\in K^0.$$
Let $\CS_K$ be the divisorial algebra associated to $K$. Denote by $I$ the ideal of $\mathcal{S}_K$ generated by the functions $1 - \chi(E)$,  $E$ runs through $K^0$, where $1$ is homogeneous of degree zero and $\chi(E)$ is homogeneous of degree~$-E$. The \emph{Cox ring} is the quotient ring $\CR_{K, \chi} = \CS_K/I$. 
The algebra $\CS_K$ is $\Cl (X)$-graded and the ideal $I$ is homogeneous with respect to this grading. So the Cox ring is also $\Cl (X)$-graded. 

The definition of the Cox ring is independent of the choice of $K$ and $\chi$. So we denote the Cox ring of $X$ simply by $\CR (X)$.

\end{de}

Let $X$ be a normal affine variety with only constant invertible functions and a finitely generated divisor class group. Suppose that the Cox ring $\CR(X)$ is also finitely generated. Consider the \emph{total coordinate space} $\overline{X} = \Spec \CR (X)$.
Since $\CR (X)$ is $\Cl (X)$-graded, we have a corresponding action of the  \emph{Neron-Severi} quasitorus $N(X) = \Spec \KK[\Cl (X)]$ on $\overline{X}$. The total coordinate space $\overline{X}$ is a normal irreducible affine variety. The variety $X$ is isomorphic to the categorical quotient $\overline{X}/\!/ N(X)$. The quotient morphism $\pi\colon \overline{X}\rightarrow X$ is called {\it the Cox realization} of $X$. 

A normal projective variety with finitely generated Cox ring is called a {\it Mori dream space}. We introduce the following terminology. A normal affine variety with  finitely generated Cox ring will be called an {\it affine Mori dream space} (AMDS).

\begin{de}\cite[Definition~4.2.3.1]{ADHL}
Let $X$ be an AMDS. Let $\pi\colon \overline{X}\rightarrow X$ be the Cox realization of $X$. Given an algebraic group action $\mu\colon G\times X\rightarrow X$.  We say that a finite epimorphism $\varepsilon\colon G'\rightarrow G$ and an action $\mu'\colon G'\times \overline{X}\rightarrow \overline{X}$ {\it lift} the $G$-action to the Cox realization, or call $(\varepsilon, \mu' )$ a {\it lifting} of $\mu$ if

1) the actions of $G'$ and $N(X)$ on $\overline{X}$ commute, 

2) the epimorphism $\varepsilon$ satisfies  $\pi(g'\cdot \overline{x})=\varepsilon(g')\cdot \pi(\overline{x})$ for all $g'\in G'$, $\overline{x}\in\overline{X}$.
\end{de}

We need the following assertion, which is a particular case of \cite[Theorem~4.2.3.2]{ADHL}.

\begin{prop}\label{AD}
Let $X$ be an AMDS. Let $\pi\colon \overline{X}\rightarrow X$ be the Cox realization of $X$.  Let $G$ be a connected affine algebraic group acting on $X$. Then 

1) There exist a finite epimorphism $\varepsilon\colon G'\rightarrow G$ of connected affine algebraic groups and an action $\mu'\colon G'\times \overline{X}\rightarrow\overline{X}$ that lift the $G$-action to the Cox realization.

2) Given a finite epimorphism $\varepsilon\colon G'\rightarrow G$  and two $G'$-actions $\circ$ and~$*$ on $\overline{X}$ that provide liftings of $G\times X\rightarrow X$, there is a homomorphism $\eta\colon G'\rightarrow N(X)$ with $g'*\overline{x}=\eta(g')(g'\circ \overline{x})$ for all $g'\in G'$ and $\overline{x}\in\overline{X}$.

3) If $G$ is a simply connected semisimple group then there is a lifting $(\mathrm{id}_G,\mu')$ of the $G$-action to the Cox realization.
 
4) If $G$ is a torus, then there is a lifting $(\varepsilon,\mu' )$ of the $G$-action to the Cox realization with $\varepsilon\colon G\rightarrow G, g\mapsto g^b$ for some $b\in\mathbb{Z}_{\geq 1}$; if in addition $\Cl(X)$ is free, then one may choose $b = 1$.
\end{prop}

\subsection{Flexibility} 
There is a relationship between flexibility of a variety and $\SAut(X)$-invariants on it.
Recall that the {\it Makar-Limanov invariant} $\ML(X)$ of an affine variety $X$ is the intersection of the kernels of all locally nilpotent derivations on X. 
In other words, $\ML(X)$ is the subalgebra of $\KK[X]$ consisting of all $\SAut(X)$-invariants. The subfield $\FML(X)$ in the field $\KK(X)$ of rational functions on $X$ consisting of all rational $\SAut(X)$-invariants is called {\it field Makar-Limanov invariant}. 

\begin{prop}\label{fml}\cite[Proposition 5.1.]{AFKKZ}
An irreducible affine variety $X$ possesses a flexible point if and only if the group $SAut(X)$ acts on $X$ with an open orbit and if and only if the field Makar-Limanov invariant $\FML(X)$ is trivial, i.e. $\FML(X)=\KK$. 
\end{prop}

The following lemma is well known, see for example \cite{Dz}.

\begin{lem}\label{D}
Let $X$ be an irreducible normal affine variety. Suppose that $\FML(X)\neq \KK$. Then there exists an $\SAut(X)$-invariant prime divisor $D\subseteq X$.
\end{lem}
\begin{proof}
Suppose that $f\in\mathrm{FML}(X)\setminus \KK$. Then $\div(f)$ is an $\SAut(X)$-invariant divisor. Let $\div(f)=\sum_{i=1}^na_iD_i$ be the decomposition of $\div(f)$ into prime divisors ($a_i\neq 0$). Let $\Omega\subseteq \mathrm{SAut}(X)$ be a subgroup isomorphic to $\G_a$. Then every $\omega\in\Omega$ permutes the divisors $D_i$, which yields the homomorphism $\tau\colon\Omega\rightarrow \mathrm{S}_n$. Since $\Omega\cong \G_a$ is a connected group, we obtain that $\tau$ is trivial. This implies that each $D_i$ is $\Omega$-invariant.
\end{proof}

\subsection{Horospherical varieties}

We recall some results on horospherical varieties, all proofs can be found in \cite{PV}, see also \cite{T}.

Let $G$ be a connected linear algebraic group. 

\begin{de}
An irreducible $G$-variety $X$ is called \emph{horospherical}, if for a generic point $x\in X$ the stabilizer of $x$ contains a maximal unipotent subgroup $U\subseteq G$.

If $X$ contains an open $G$-orbit, then $X$ is called \emph{complexity-zero horospherical}. In \cite{PV} affine complexity-zero horospherical varieties are called $S$-varieties.
\end{de}

Suppose that $X$ is an affine complexity-zero horospherical variety. It is easy to see that the unipotent radical of $G$ acts trivially on $X$. Hence we may assume that $G$ is reductive. Taking a finite covering, we may assume that $G=T\times G'$, where $T$ is an algebraic torus and $G'$ is a semisimple group.

Let $O$ be the open orbit in $X$. We have the following sequence of inclusions
$$
\KK[X]\hookrightarrow\KK[O]\hookrightarrow\KK[G].
$$

Let $B$ be a Borel subgroup of $G$ and let $\mathfrak{X}(B)$ be the group of characters of $B$.   
For a $\Lambda\in\mathfrak{X}(B)$ we put
$$
S_\Lambda=\left\{ f\in\KK[G]\mid f(gb)=\Lambda(b)f(g)\ \text{for all} \ g\in G, b\in B\right\}.
$$
Then
$$
S_{\Lambda} S_{M}=S_{\Lambda+M}.
$$

The set $\mathfrak{X}^+(B)$ of dominant weights consists of all $\Lambda$ such that $S_\Lambda\neq\{0\}$. 
It is proved in \cite{PV} that for an affine complexity-zero horospherical $G$-variety $X$ there is a decomposition
$$
\KK[X]=\bigoplus_{\Lambda\in P} S_\Lambda
$$
for some subsemigroup $P\in\mathfrak{X}^+(B)$. 

Let us consider a $\mathbb{Q}$-vector space $V=\mathfrak{X}^+(B)\otimes_\mathbb{Z} \mathbb{Q}$. 
Denote by $\sigma$ the cone in $V$ spanned by all elements of the semigroup $P$. The variety $X$ is normal if and only if $P$ is saturated, i.e. 
$\mathbb{Z}P\cap \sigma =P$. There is a one-to-one correspondence between faces of $\sigma$ and $G$-orbits of $X$. More precisely, if
$Y\subseteq  X$ is a $G$-orbit in $X$ then there is a corresponding face $\tau$ of the cone $\sigma$ such that the ideal of functions vanishing on $Y$ has the form 
$$
I(Y)=\bigoplus_{\Lambda\in P\setminus\tau}S_\Lambda.
$$

\section{Non-hyperbolic $\G_m$-actions and locally nilpotent derivations}

Let $X$ be a normal affine irreducible variety. Regular actions of a one-dimensional torus $\G_m$ on $X$ correspond to $\ZZ$-gradings of the algebra of regular functions
$$\KK[X]=\bigoplus_{i\in\ZZ}\KK[X]_i.$$ 
If there exist $i>0$ and $j<0$ such that both homogeneous components $\KK[X]_i$ and $\KK[X]_j$ are nonzero, then the $\G_m$-action is called {\it hyperbolic}. Suppose a $\G_m$-action is not hyperbolic. Then we may assume $\KK[X]=\bigoplus_{i\geq 0}\KK[X]_i$. If $\KK[X]_0=\KK$, then the $\G_m$-action is {\it elliptic}. If $\KK[X]_0\neq\KK$, then the $\G_m$-action is {\it parabolic}. Further all elliptic and parabolic actions are called  {\it non-hyperbolic}. We assume that $\KK[X]_i=\{0\}$ for all $i<0$ in this case.

Let us concider a non-hyperbolic $\G_m$-action. We have a natural $\G_m$-invariant ideal $I=\bigoplus_{i>0}\KK[X]_i$ in $\KK[X]$. Let $Z$ be the zero locus of $I$. Then for every $x\in X$ there exists $\lim_{t\rightarrow 0}t\cdot x\in Z$.  The variety $Z$ is the set of $\G_m$-fixed points. 

\begin{prop}\label{po}
Let $Z$ be the set of $\G_m$-fixed points for a non-hyperbolic $\G_m$-action on a normal affine irreducible variety $X$. Assume $Z\cap X^{\reg}\neq \varnothing$. Then $Z^{\reg}\cap X^{\reg}\neq \varnothing$ and for every $z\in Z^{reg}\cap X^{\reg}$ the tangent space $T_zX$ is spanned by $T_zZ$ and tangent vectors to orbits of regular $\mathbb{G}_a$-actions on~$X$.
\end{prop}
\begin{proof}
We have $\KK[Z]\cong\KK[X]_0$. Therefore, $Z$ is irreducible. Two nonempty open subsets $X^{\reg}\cap Z$ and $Z^{\reg}$ have nonempty intersection. 
Let $z$ be a point in $Z^{\reg}\cap X^{\reg}$. Denote by $\mathfrak{m}_z$ the maximal ideal in $\KK[X]_0\cong\KK[Z]$ corresponding to $z$. Then the maximal ideal in $\KK[X]$ corresponding to $z$ is $\mathfrak{M}_z=\mathfrak{m}_z\oplus I(Z)$. Denote $\dim Z=k$, $\dim X=k+n$.  Since $z$ is a regular point of $Z$, we have $k=\dim Z=\dim (\mathfrak{m}_z/\mathfrak{m}_z^2)$. Let $h_1,\ldots,h_{k}\in \mathfrak{m}_z$ be functions such that their images in $\mathfrak{m}_z/\mathfrak{m}_z^2$ give a basis. Since $z$ is a regular point of $X$, we obtain
$$
n+k=\dim X=\dim \left(\mathfrak{M}_z/\mathfrak{M}_z^2\right).
$$ 
We have 
$$
\mathfrak{M}_z/\mathfrak{M}_z^2=(\mathfrak{m}_z\oplus I(Z))/(\mathfrak{m}_z\oplus I(Z))^2=\mathfrak{m}_z/\mathfrak{m}_z^2\oplus \left(I(Z)/(\mathfrak{m}_zI(Z)+I(Z)^2)\right).
$$ 
Therefore, we can take $f_1,\ldots,f_n\in I(Z)$ such that images of $\{h_1,\ldots h_k, f_1,\ldots, f_n\}$ give a basis of $\mathfrak{M}_z/\mathfrak{M}_z^2$.  
We can assume $f_1,\ldots, f_n$ to be $\ZZ$-homogeneous.
Then $h_1,\ldots h_k, f_1,\ldots,f_n$ are algebraically independent. Hence, $\{h_1,\ldots h_k\}$ is a transcendence basis of $\KK[X]_0$ and $f_1,\ldots,f_n$ are algebraically independent over $\KK[X]_0$. Let $\mathcal{A}=\KK[X]_0[f_1,\ldots,f_n]$. The $\ZZ$-grading of $\KK[X]$ induces a $\ZZ$-grading on $\mathcal{A}$.
\begin{lem}\label{prd}
Let $\delta$ be a homogeneous LND of $\mathcal{A}$ such that $\deg( \delta)<0$. Then there exists $s\in \KK[X]_0$ such that $s\delta$ can be extended to an LND of $\KK[X]$.
\end{lem}
\begin{proof}
Since $\KK[X]$ is an algebraic extension of $\mathcal{A}$, the derivation $\delta$ can be extended to a derivation $\KK[X]\rightarrow\KK(X)$, see \cite[Section~2.3]{VO}. 

Let $g_1,\ldots, g_m$ be homogeneous generators of $\KK[X]$. 
Let us prove that for each non negative integer $u$ there is $s_u\in\KK[X]_0$ such that $s_u\delta(g_i)\in\KK[X]$ for all $g_i$ in $\KK[X]_j$, where $j\leq u$. We proceed by induction on $u\in\ZZ_{\geq 0}$. 

{\it Base of induction, $u=0$.}
If $g_i\in\KK[X]_0$, then $\delta(g_i)=0$. So, we can take $s_0=1$.

{\it Step of induction.} Denote $\rho=s_{u-1}\delta$. If $g_i\in\KK[X]_u\subseteq I(Z)$, we have 
\begin{equation}\label{ee}
g_i=a_1f_1+\ldots+a_nf_n+F_i +G_i,
\end{equation}
where $a_j\in\KK$, $F_i\in \mathfrak{m}_z I(Z)$, $G_i\in I(Z)^2$. Let $g_{i_1},\ldots,g_{i_l}$ be all $g_i$ in $\KK[X]_u$. Then we can assume $F_i=p_1g_{i_1}+\ldots+p_lg_{i_l}$, where $p_j\in \mathfrak{m}_z$. Using equations of the form (\ref{ee}) we obtain 
$$
g_{i_l}=\frac{\sum_{j=1}^na_jf_j+F_{i_l}+p_1g_{i_1}+\ldots+p_{l-1}g_{i_{l-1}}}{1-p_l}. 
$$
Let us consider the localization $\mathcal{B}=\KK[X]_{(\KK[X]_0\setminus\mathfrak{m}_z)}$. We have

$$
g_{i_l}=\sum_{j=1}^n\tilde{a}_jf_j+\tilde{G_{i_l}} +\sum_{t=1}^{l-1}\tilde{p}_tg_{i_t},
$$
where $\tilde{a}_j=\frac{a_j}{1-p_l} \in\mathcal{B}$, $\tilde{G_{i_l}}=\frac{G_{i_l}}{1-p_l}\in I(Z)^2_{(\KK[X]_0\setminus \mathfrak{m}_z)}$, $\tilde{p_t}=\frac{p_t}{1-p_l}\in {\mathfrak{m}_z}_{(\KK[X]_0\setminus \mathfrak{m}_z)}$.

Then we can substitute this to all the other equations of the form~(\ref{ee}). And so on until we obtain an equation with only one $g_{i_1}$.
$$g_{i_1}=\sum_{j=1}^n \widehat{a}_jf_j+\widehat{G},$$ 
where $\widehat{a}_j\in\mathcal{B}$,  $\widehat{G}\in I(Z)^2_{(\KK[X]_0\setminus \mathfrak{m}_z)}$. 
Hence,  
$$g_{i_1}=\frac{\sum_{j=1}^n \overline{a}_jf_j+\overline{G}}{P},$$
where $\overline{a}_j\in\KK[X]_0$, $\overline{G}\in I(Z)^2$, $P\in\KK[X]_0\setminus \mathfrak{m}_z$.
Therefore, 
$$
\rho(g_{i_1})=\frac{\sum_{j=1}^n \overline{a}_j\rho(f_j)+\rho(\overline{G})}{P}.
$$
Hence, $P\rho(g_{i_1})\in\KK[X]$. Analogically we obtain functions $P_1=P, P_2,\ldots P_l$ in $\KK[X]_0\setminus \mathfrak{m}_z$ such that $P_j\rho(g_{i_j})\in\KK[X]$. Then we can take $s_u=\left(\prod_{j=1}^l P_j\right)s_{u-1}$. 

Since the number of $g_i$ is finite, we can take 
$$u_0=\max\{u\mid \exists g_i\in \KK[X]_{u_0}\}.$$ 
Then for $s=s_{u_0}$ there exists a derivation $\partial$ of $\KK[X]$, which extends $s\delta$. It is easy to see, that $\deg \partial=\deg\delta$.
Since $\deg \partial<0$, it is locally nilpotent. Lemma~\ref{prd} is proved. 
\end{proof}

Let us consider $n$ LNDs 
$\delta_1=\frac{\partial}{\partial f_1},\ldots, \delta_n=\frac{\partial}{\partial f_n
}$ of $\mathcal{A}$.
We have $n$ LNDs $\partial_1,\ldots, \partial_n$ on $\KK[X]$, which are extensions of derivations $s_1\frac{\partial}{\partial f_1}, \ldots, s_n \frac{\partial}{\partial f_n}$ of $\mathcal{A}$.  Let us denote $\varphi_j(t)=\exp(t\partial_j)$. We have
$$
T_zX\cong \left(\mathfrak{M}_z/\mathfrak{M}_z^2\right)^*=\langle h_1+\mathfrak{M}_z^2,\ldots, h_k+\mathfrak{M}_z^2,f_1+\mathfrak{M}_z^2,\ldots, f_n+\mathfrak{M}_z^2\rangle^*.
$$ 
By definition $\partial_j(h_p)=0$, $\partial_j(f_j)=s_j$, $\partial_j(f_i)=0$ for all $i\neq j$. Since $s_j\in\KK[X]_0\setminus \mathfrak{m}_z$, we obtain $\partial_j(h_p)(z)=0$, $\partial_j(f_j)(z)\neq 0$, $\partial_j(f_i)(z)=0$ for all $i\neq j$.
That is tangent vector of $\varphi_j(t)\cdot z$ is proportional to the dual basis vector to $f_j+\mathfrak{M}_z^2$. Therefore, dual basis vectors to $h_i+\mathfrak{M}_z^2$ are in $T_zZ$ and dual basis vectors to $f_j+\mathfrak{M}_z^2$ are  tangent vectors to $\mathbb{G}_a$-orbits. Proposition~\ref{po} is proved.
\end{proof}

Proposition \ref{po} implies the following two assertions. 

\begin{cor}\label{sle}
Under the assumptions of Proposition~\ref{po}, the set $Z$ is not $\SAut(X)$-invariant.
\end{cor}

\begin{cor}\label{co}
Under the assumptions of Proposition~\ref{po}, assume $z\in Z^{\reg}\cap X^{\reg}$. If the tangent space $T_zZ$ is generated by tangent vectors to $\mathbb{G}_a$-orbits for algebraic $\mathbb{G}_a$-actions on $X$, then $z$ is a flexible point of $X$.
\end{cor}

\section{Sufficient conditions of possessing a flexible point}

For many important classes of varieties we know some natural automorphisms of each variety for this class. In typical situation this set of natural automorphisms contains some $\G_a$-actions and some $\G_m$-actions. Often known $\G_a$-actions on a variety $X$ are not sufficient to prove possessing a flexible point. In this section we elaborate some technique, which allows to prove possessing of a flexible point using not only $\G_a$-actions but also an action of an algebraic torus, i.e. commuting set of $\G_m$-actions.

More precisely, we prove that if the group $H$ generated by $\SAut(X)$ and a torus $T$ acts on an AMDS $X$ with an orbit $O$ such that codimension of $X\setminus O$ in $X$ is at least 2, then $X$ possesses a flexible point. The idea is to find an $H$-semi-invariant nonconstant function in condition that $\FML(X)\neq\KK$. If $\ML(X)\neq\KK$ we can consider the $T$-action on $\ML(X)$ and find a nontrivial semi-invariant. The problem is that $T$-action on $\FML(X)$ a priori can be non regular. We use Cox rings to overcome this difficulty.

\begin{prop}\label{main}
Assume $T$ is a torus acting on an AMDS $X$. Let $H$ be the subgroup in $\Aut(X)$ generated by $T$ and $\SAut(X)$. Assume $X$ does not possess a flexible point. Then there is an $H$-invariant prime divisor $D\subseteq X$.
\end{prop}

\begin{proof}
Since $\CR(X)$ is finitely generated, we can consider the total coordinate space $\overline{X}$. Let us denote by $\widetilde{\mathrm{Aut}}(\overline{X})$ the group of automorphisms of $\overline{X}$ that normalize the subgroup $N(X)\hookrightarrow \mathrm{Aut}(\overline{X})$.

By \cite[Theorem~5.1]{AG} we have the following exact sequence

\begin{equation}\label{exact}
1 \longrightarrow N(X) \stackrel{\alpha}{\longrightarrow} \widetilde{\mathrm{Aut}}(\overline{X}) \stackrel{\beta}{\longrightarrow} \Aut(X) \longrightarrow 1
\end{equation}

By Proposition~\ref{AD}(4)  there is a $T$-action on $\overline{X}$ lifting the $T$-action on $X$. That is, we have a subgroup $S\cong T$ in $\beta^{-1}(T)\subset\widetilde{\mathrm{Aut}}(\overline{X})$.

Since $X$ does not possess a flexible point, by Lemma~\ref{D} there is a prime $\mathrm{SAut}(X)$-invariant divisor $D_0\subseteq X$.

{\it Case 1} The divisor $D_0$ is principle. Then $D_0=\div(f)$, where $f$ is $\mathrm{SAut}(X)$-semi-invariant. Since $\mathrm{SAut}(X)$ is generated by unipotent subgroups, each $\mathrm{SAut}(X)$-semi-invariant is invariant. Therefore, $f\in\ML(X)\setminus \KK$. Since $\SAut(X)$ is a normal subgroup of $\Aut(X)$, $\ML(X)$ is $\Aut(X)$-invariant subring, in particular $\ML(X)$ is $T$-invariant. Then $\ML(X)$ is a direct sum of $T$-weighted spaces. Hence, there is a $T$-semi-invariant $p\in\ML(X)\setminus \KK$. Then $\div (p)$ is an $H$-invariant divisor. Then each prime divisor in $\supp(\div(g))$ is $H$-invariant. 

{\it Case 2.} The divisor $D_0$ is not principle. Consider a finitely generated group $K\subseteq  \WDiv (X)$ from Definition~\ref{coxring}. We can assume $D_0\in K$.  Let us take
$$f=1\in\mathcal{L}(X,D_0).$$

Consider the image $\overline{f}$ of $f$ in $\CR (X) = \KK[\overline{X}]$. Since $D_0$ is not principle, we have that $\overline{f}\notin \KK$. Let us consider $\Gamma=\beta^{-1}(\mathrm{SAut}(X))$.

\begin{lem}\label{Ginv}
The element $\overline{f}$ is $\Gamma$-semi-invariant.
\end{lem}

\begin{proof}
Any automorphism $\psi \in \SAut (X)$ defines the automorphism of the group $\WDiv (X)$.  The group $\psi(K)$ is also finitely generated and the map $\psi(K) \to \Cl(X)$ is surjection. The dual map $\psi^*$ maps the space $\CL(X, D)$ to the space $\CL(X, \psi(D))$ where $D \in \WDiv (X)$. So we have a homomorphism of algebras 
$\CS_K \to \CS_{\psi(K)}$ sending $\overline{f}\in\CR(X)$ to itself. At homogeneous elements this homomorphism coincides with $\psi^*$.

Consider a character $\chi$ as in Definition \ref{coxring}. It is easy to see that the mapping $\psi^*$ induces a mapping of Cox rings $\CR_{K, \chi}$ and $\CR_{\psi(K), \psi^*\circ\chi}$. But the Cox rings 
$\CR_{K, \chi}$ and 
$\CR_{\psi(K), \psi^*\circ\chi}$ are isomorphic as $\Cl(X)$-graded algebras. Moreover, there is an isomorphism of graded algebras
$\tau \colon \CR_{\psi(K), \psi^*\circ\chi} \to \CR_{K, \chi}$ such that its restriction to $\CR_0=\KK[X]$ is the identity mapping. Since $\psi(D_0) = D_0$ the isomorphism $\tau$ multiplies the space $\CL(X, D_0)$ by some element 
of~$\KK\setminus \{0\}$. The composition of $\psi$ and $\tau$ gives an automorphism $\overline{\psi}$ of $\CR_{K, \chi}$ which is equal to $\psi$ on $\CR_{K, \chi} (X)_0$ and 
$\overline{\psi} (\overline{f}) = c\overline{f}$ where $c \in \KK\setminus \{0\}$. So $\overline{f}$ is a semi-invariant with respect to the automorphism~$\overline{\psi}$. 

It is easy to see that $\overline{f}$ is a semi-invariant with respect to the action of the group $N(X)$. But it follows from exact sequence (\ref{exact}) that the group $\Gamma$ is generated by the subgroup $N(X)$ and automorphisms $\overline{\psi}$ where $\psi\in \SAut(X)$. So $\overline{f}$ is a semi-invariant with respect to the group $\Gamma$.
\end{proof}

It follows from Lemma \ref{Ginv} that $\CR(X)^{(\Gamma)}\neq \KK$.
Let $\CR(X)_{\xi}$ be a $\Gamma$-weighted subspace corresponding to $\xi\in\mathfrak{X}(\Gamma)$.
\begin{lem}
The subspace $R(X)_{\xi}$ is $T$-invariant.
\end{lem}
\begin{proof}
Since $\SAut(X)$ is normal in $\Aut (X)$ then $\Gamma$ is normal in $\widetilde{\mathrm{Aut}} (\overline{X})$. Let $h\in R(X)_{\xi}$, $t\in T$, $g\in \Gamma$. Then
$$
 g\cdot (t\cdot h)=t\cdot (t^{-1}\cdot  (g\cdot (t\cdot h)))=t\cdot (\hat{g}\cdot h)=t\cdot\xi(\hat{g})h=\xi(\hat{g})(t\cdot h),$$
where $\hat{g}$ is an element of $\Gamma$. Therefore, $t\cdot h$ is a $\Gamma$-semi-invariant. We obtain a homomorphism $T\rightarrow \Aut(\mathfrak{X}(\Gamma))$. Since $T$ is connected, this homomorphism is trivial.
\end{proof}

The action of the group $T$ on $R(X)$ is rational. Therefore, for every $\xi\in\mathfrak{X}(\Gamma)$, $R(X)_\xi$ is a direct sum of $T$-weighted spaces. Since $R(X)^{(\Gamma)}\neq \KK$, there is $q\in R(X)\setminus \KK$, which is a $\Gamma$-semi-invariant and a $T$-semi-invariant.

Let us consider the divisor $\div(q)\subseteq \overline{X}$. This divisor is $\Gamma$-invariant. Hence, it is $N(X)$-invariant.  Therefore, $\pi(\mathrm{supp}(\div(q)))$ is a divisor.
Consider a prime divisor $\overline{D}$ in $\mathrm{supp}(\div(q))$. The divisor $D=\pi(\overline{D})$ lies in the support of the divisor $\pi(\mathrm{supp}(\div(q)))$. Since $H$ is generated by connected algebraic subgroups, $D$ is $H$-invariant.

Proposition \ref{main} is proved.
\end{proof}

\begin{lem}
Let $X$ be an AMDS. Suppose $H$ acts transitively on $X^{\reg}$. Then $X$ is flexible.
\end{lem}
\begin{proof}
Suppose $X$ has no flexible points. By Proposition~\ref{main} there exists an $H$-invariant prime divisor $D$ on $X$. Since $X$ is normal, $D\cap X^{\reg}\neq \emptyset$. Therefore, we obtain a contradiction. Hence, there is a flexible point on $X$. But since $H$ acts on $X^{\reg}$ transitively, we obtain that all regular points on $X$ are flexible, that is $X$ is flexible.
\end{proof}

\section{Varieties with a linear algebraic group action}

In this section we use technique of the previous one to prove flexibility of some classes of varieties. To prove flexibility of a variety we need some natural automorphisms of $X$. An action of a linear algebraic group often gives these automorphisms.

\begin{prop}
Let $X$ be an AMDS. Let a linear algebraic group $G$ acts on $X$ with an open orbit $O$. Suppose $\mathrm{codim}_{X}X\setminus O\geq 2$. Then $O$ consists of flexible points. 
\end{prop} 
\begin{proof}
Denote by $T$ the maximal torus in $G$.
Let $H$ be the subgroup in $\Aut(X)$, generated by  $\SAut(X)$ and the image of $T$ in $\Aut (X)$. We have $G\subseteq H$.

Suppose $X$ does not possess a flexible point. Then by Proposition~\ref{main} there exists an $H$-invariant prime divisor $D\subseteq X$. Since $G\subseteq H$, the divisor $D$ is $G$-invariant. A contradiction. Hence, $X$ possesses a flexible point. Therefore, by \cite[Corollary~1.11]{AFKKZ} there exists an open subset $U\subseteq X$ consisting of flexible points. So there is a flexible point in $O$. Therefore, all points in $O$ are flexible.
\end{proof}

If $G$ is a reductive group, then the condition that $X$ is AMDS is always satisfied. Indeed, the following assertion extends results due to Knop \cite{K} to Cox rings.
\begin{prop}\cite[Theorem~4.3.1.5]{AFKKZ}\label{rg}
Let $G$ be a connected reductive algebraic group and $X$ be an irreducible normal unirational $G$-variety of complexity $c(X)\leq 1$ with only constant invertible global functions. Then the divisor class group $\mathrm{Cl}(X)$ and the Cox ring $\mathcal{R}(X)$ are finitely generated.
\end{prop} 
If $G$ acts on $X$ with an open orbit, then $X$ is rational and hence it is unirational. So, each irreducible normal affine variety $X$ admitting only constant invertible functions with locally transitive action of a reductive group $G$ is AMDS.

\begin{theor}\label{sm}
Let $X$ be a smooth irreducible affine variety with only constant invertible regular functions. Assume a reductive algebraic group $G$ acts on $X$ with an open orbit. Then $X$ is flexible.
\end{theor}
\begin{proof}
Let $Z$ be the unique closed orbit in $X$. The stabilizer of a point on this orbit is a reductive subgroup $S\subseteq  G$, see \cite[Theorem~4.17]{VPT}. 
Moreover, it follows from Luna's Etale Slice Theorem that there is a finite dimensional rational $S$-module $W$ such that the variety $X$ is $G$-equivariantly isomorphic to the total space $G*_SW$ of the homogeneous vector bundle over $G/S$ with the fiber $W$, see \cite[Theorem 6.7]{VPT}.

Consider an elliptic $\G_m$-action on $W$ via homoteties. It commutes with the $S$-action. Therefore, it induce a non-hyperbolic action on $X$ with the set of fixed points equal to~$Z$. Denote by $T$ the image of $\G_m$ in $\Aut(X)$. Let $H$ be the group generated by $\SAut(X)$ and~$T$. Let $z\in Z$. By Proposition~\ref{po} the tangent space $T_zX$ is spanned by $T_zZ$ and tangent vectors to $\G_a$-orbits. 

Suppose $X$ does not possess a flexible point. Then by Proposition~\ref{main} there is an $H$-invariant divisor $D\subseteq X$. Therefore, $D$ is $G$-invariant. Since $Z$ is the unique closed $G$-orbit in $X$, we have $Z\subseteq  D$. Hence, $T_zZ\subseteq T_zD$. Since $D$ is $\SAut(X)$-invariant, we have that all tangent vectors to $\G_a$-orbits at $z$ lies in $T_zD$. Therefore, $T_zZ$ and tangent vectors to all $\G_a$-orbits of $z$ spans a subspace in $T_zD$. A contradiction. Hence, $X$ possesses a flexible point.

By Proposition~\ref{fml} the group $\SAut(X)$ acts on $X$ with an open orbit $O$. This open orbit consists of all flexible points on $X$, see \cite[Corollary~1.11]{AFKKZ}. Suppose, $O\neq X$. Then $D=X\setminus O$ is a nonempty closed 
$\Aut(X)$-invariant subset. In particular, $D$ is $H$-invariant. This gives a contradiction as above.

\end{proof}

\begin{ex}
Consider the hypersurface $X$ in $\KK^5$ given by
$$
y_{11}x_1^2+2y_{12}x_1x_2+y_{22}x_2^2=1,
$$
It can be written in the following way
$$
\begin{pmatrix}
x_1&x_2
\end{pmatrix}
\begin{pmatrix}
y_{11}&y_{12}\\
y_{12}&y_{22}
\end{pmatrix}
\begin{pmatrix}
x_1\\
x_2\\
\end{pmatrix}
=1
$$
It is easy to see that $X$ is smooth. We can take a symmetric bilinear functions $f$ on a vector space $\KK^2$ with matrix $(y_{ij})$ and a vector $v=(x_1,x_2)$. Then $f(v, v)=1$. The group $\mathrm{GL}_2(\KK)$ acts on the set of pairs $(f, v)$ by linear changing of basis. There are two $\mathrm{GL}_2(\KK)$-orbits on $X$. Orbit consisting of pairs $(f, v)$  with non-degenerate $f$ is an open one. Therefore, $X$ is flexible.
\end{ex}

The condition that $G$ is reductive cannot be waived in Theorem~\ref{sm}. Let us give an example of a smooth irreducible affine variety having only constant invertible functions with a locally transitive action of $\G_m\rightthreetimes \G_a$ which is not flexible.
\begin{ex}\label{ce}
Let $X=\left\{xy^2=z^2-1\right\}$ be a Danielewski surface. It is easy to see that $X$ is smooth. The Makar-Limanov invariant of $X$ equals $\KK[y]$, see \cite{ML}, therefore, $X$ is not flexible. If $X$ admits a nonconstant invertible function $f$, then $f\in\ML(X)=\KK[y]$. Therefore, $f(y)g(y)=1$. This is not possible if $f$ or $g$ is not a constant. So, $X$ is a smooth irreducible affine variety with only constant invertible functions, which is not flexible.
Consider the following action of $\G_m\rightthreetimes \G_a$ on $X$
$$
(t,s)\cdot(x,y,z)
=(t^2(x+2zs+s^2y^2),t^{-1}y,z+sy^2).
$$
The set $O=\{p\in X\mid y(p)\neq 0\}$ is the open orbit. 

\end{ex}

\begin{prop}\label{ot}
Suppose a linear algebraic group $G$ acts on an AMDS $X$. Assume for each $G$-orbit $O\subseteq X$ consisting of regular points there is a non-hyperbolic $\G_m$-subgroup $\Lambda$ in $\Aut(X)$ such that the set of $\Lambda$-fixed points lies in the closure $\overline{O}$ and contains a regular point. Then $X$ is flexible.
\end{prop}
\begin{proof}
Suppose $X$ does not possess a flexible point. Let $T$ be a maximal torus in $G$. Denote by $H$ the group generated by $\SAut(X)$ and $T$. Then by Proposition~\ref{main} there exists an $H$-invariant prime divisor $D$. Since $X$ is normal, codimension of the set of singular points is at least 2. Therefore, $D$ contains a regular point. The divisor $D$ is $G$-invariant. Therefore, $D$ contains an orbit $O$ consisting of regular points. There is a non-hyperbolic $\G_m$-subgroup $\Lambda$ in $\Aut(X)$ such that the set of $\Lambda$-fixed points $L$ lies in the closure $\overline{O}$ and contains a regular point. Fix a regular point $p\in L$. All tangent vectors to $\G_a$-orbits at $p$ lies in $T_pD$. This contradicts to  Proposition~\ref{po}. Hence, $X$ possesses a flexible point.

Let $U$ be the set of all flexible points of $X$. It is an open nonempty subset of $X$, see \cite[Proposition~1.11]{AFKKZ}. Suppose there exists a regular point $p\in X\setminus U$. Then $D=X\setminus U$ is a proper closed $\Aut(X)$-invariant subset containing a regular point. We obtain a contradiction as above. Therefore, $X$ is flexible.  
\end{proof}

Let us use Proposition~\ref{ot} to give a new proof of flexibility of normal affine $\mathrm{SL}_2(\KK)/\mu_m$-embeddings, see \cite[Theorem 5.7]{AFKKZ} for original proof.

\begin{ex}
Let $X$ be a normal irreducible three-dimensional affine variety with a locally transitive $\mathrm{SL}_2(\KK)$-action.  Such varieties were classified by Popov in \cite{Po}. 
If $X\cong G/H$ is homogeneous, then it is flexible by  \cite[Proposition~5.4]{AFKKZ}. According to \cite{Po}, a normal non-homogeneous affine $\mathrm{SL}_2$-threefold with an open orbit is uniquely determined by a pair $(h, m)$, where $m$ is the order of generic stabilizer and $h = p/q\in(0, 1]\cap\mathbb{Q}$ is the so called {\it height} of $X$. The generic stabilizer is a cyclic group $\mu_m$. The variety $X$ corresponding to a pair $(h, m)$ is denoted by $E_{h,m}$ and is called a {\it normal affine $\mathrm{SL}_2(\KK)/\mu_m$-embedding of height~$h$}. If $h=1$, then $X$ is smooth and it is flexible by \cite[Theorem~5.6]{AFKKZ}.  

Batyrev and Haddad \cite{BH} describe the Cox realization of $E_{h,m}$ with $h<1$. Assume that $p$ and $q$ are coprime positive integers. Denote 
$$a=m/k,\qquad b=(q-p)/k,\qquad  \text{where}\ k=\gcd(q-p,m).$$

Consider a hypersurface $D_b=\{y^b=x_1x_4-x_2x_3\}$. There is an $\mathrm{SL}_2(\KK)$-action on this hypersurface 
$$
\begin{pmatrix}
\alpha& \beta\\
\gamma& \delta
\end{pmatrix}\cdot
(x_1, x_2, x_3, x_4, y)=(\alpha x_1+\beta x_2, \gamma x_1+\delta x_2, \alpha x_3+\beta x_4, \gamma x_3+\delta x_4, y).
$$

Let $\mu_a$ be a cyclic group generated by a primitive root of unity of degree $a$. Then $E_{h,m}$ is $\mathrm{SL}_2$-equivariant isomorphic to categorical quotient of the hypersurface $D_b$
modulo the diagonal action of the group $N=\G_m\times \mu_a$ 
via
$$(t,\xi)\cdot (x_1, x_2, x_3, x_4, y)=(t^{-p}\xi^{-1} x_1, t^{-p}\xi^{-1} x_2, t^{q}\xi x_3, t^{q}\xi x_4, t^ky).$$

There are three $\mathrm{SL}_2$-orbits on $E_{h,m}$ if $h<1$. Two of them, the open one and the two-dimensional one, are regular. The third one is a singular $\mathrm{SL}_2$-fixed point. The inverse image of the closure of two-dimensional orbit in $D_b$ is the set $\{y=0\}$. 

Let us consider the following $\G_m$-action on $D_b$ 
$$t\cdot (x_1, x_2, x_3, x_4, y)=(t^px_1, t^qx_2, t^{-q}x_3, t^{-p}x_4, y).$$ This action commute with $N$-action. Hence, it induces a  $\G_m$-action on~$E_{h,m}$. Let us denote the corresponding $\G_m$-subgroup of $\Aut(E_{h,m})$ by $\Lambda$.

The algebra of $N$-invariants is spanned by monomials of the form $x_1^{s}x_2^ux_3^vx_4^wy^z$, where
\begin{equation}\label{eq}
\begin{cases}
-ps-pu+qv+qw+kz=0,\\
a\mid (-s-u+v+w).
\end{cases}
\end{equation}
The $\Lambda$-weight of such a monomial equals 
$$
\tau=sp+uq-vq-wp.
$$
Using the first equation from (\ref{eq}) we obtain
$$
\tau=u(q-p)+w(q-p)+kz\geq 0.
$$
Therefore, the $\Lambda$-action on $E_{h,m}$ is non-hyperbolic. If $\tau=0$, then $z=0$, hence, the set of $\Lambda$-fixed points lies in the closure of the two-dimensional $\mathrm{SL}_2$-orbit. The point $P=(x_1,x_2,x_3,x_4,y)=(1,0,1,0,0)\in D_b$ lies in the inverse image of the set of $\Lambda$-fixed points. Consider the following $N$-invariant $f=x_1^{aq}x_3^{ap}$. We have $f(P)\neq 0$. Therefore, $P$ is not in the inverse image of the singular point.  Now we can apply Proposition~\ref{ot}, which implies flexibility of $E_{h,m}$.
\end{ex}

The following corollary is a particular case of Proposition~\ref{ot}.

\begin{cor}\label{slor}
Suppose a linear algebraic group $G$ acts on an AMDS $X$. Assume for every $G$-orbit $O\subseteq X$ consisting of regular points there is a non-hyperbolic $\G_m$-subgroup $\Lambda$ in $\Aut(X)$ such that the set of $\Lambda$-fixed points is the closure $\overline{O}$. Then $X$ is flexible.
\end{cor}

Now we use our technique to prove flexibility of a normal affine complexity-zero horospherical variety without nonconstant invertible functions.

\begin{theor}\label{las}
Let $X$ be a normal affine complexity-zero horospherical variety of a connected group $G$. Assume  $X$ does not admit any nonconstant invertible functions. Then $X$ is flexible.
\end{theor}
\begin{proof}
As we have mentioned in Section~\ref{pre}, we can assume that $G = T\times G'$, where $T$ is a torus and $G'$ is a connected semisimple group.

Let $P$ be a subsemigroup in $\mathfrak{X}(B)^+$  corresponding to $X$ and $\sigma$ be the cone spanned by $P$ in the space $V=\mathfrak{X}(B)^+\otimes_{\mathbb{Z}}\mathbb{Q}$, see Section~\ref{pre}. 

Let $O\subseteq X$ be a $G$-orbit consisting of regular points. 
Then $O$ corresponds to a face $\tau$ of $\sigma$. There is a linear function $l$ on $V$ such that $l\mid_{\tau}=0$ and $l$ has positive integer values at all elements of $P\setminus \tau$. Let us consider a $\mathbb{Z}$-grading of $\KK[X]$ given by $l$, that is 
$$
\KK[X]_i=\bigoplus_{p\in P, l(p)=i} S_p.
$$ 
We have $I(\overline{O})=\bigoplus_{i>0}\KK[X]_i$. This $\mathbb{Z}$-grading corresponds to a non-hyperbolic $\G_m$-action. The set of $\G_m$-fixed points is the set of zeros of $\bigoplus_{i>0}\KK[X]_i$, i.e. $\overline{O}$. Since $X$ admits a locally transitive action of a reductive group by Proposition~\ref{rg} it is an AMDS. 
By Corollary~\ref{slor} $X$ is flexible. 

\end{proof}

\end{document}